\documentclass[reqno]{amsart}
\usepackage{amsmath,amsthm,amssymb}
\usepackage{latexsym}
\usepackage{eucal}
\usepackage{fullpage}

\newtheorem{theorem}{Theorem}[section]
\newtheorem{lemma}{Lemma}[section]

\newtheorem{proposition}{Proposition}[section]

\theoremstyle{definition}

\def\ge{\geqslant}\def\le{\leqslant}
\def\~{\widetilde}

\begin{document}
 \title[ On the growth of the support of 
 positive 
 vorticity \\
for 2D Euler in an infinite cylinder.
 ]{  On the growth of the support of 
 positive 
 vorticity \\
 for 2D Euler equation in an infinite cylinder}
\author{K. Choi, S. Denisov  }

\address{
\begin{flushleft}
\vspace{1cm} Kyudong Choi: kchoi@unist.ac.kr\\ \vspace{0.1cm}
Ulsan National Institute of Science and Technology  \\
Department of Mathematical Sciences \\
UNIST-gil 50, Ulsan, 44919,
Republic of Korea\\
\vspace{1cm} Sergey Denisov: denissov@wisc.edu\\\vspace{0.1cm}
University of Wisconsin--Madison\\  Mathematics Department\\
480 Lincoln Dr., Madison, WI, 53706,
USA\vspace{0.1cm}\\and\\\vspace{0.1cm}
Keldysh Institute for Applied Mathematics, Russian Academy of Sciences\\
Miusskaya pl. 4, 125047 Moscow, RUSSIA\\
\end{flushleft}
}

\renewcommand{\thefootnote}{\fnsymbol{footnote}}
\footnotetext{\emph{Key words:} 2D Euler, diameter of support, positive vorticity,
confinement, large time behavior, an incompressible flow, an infinite strip,
 a two-dimensional cylinder
\quad\emph{2010 AMS Mathematics Subject Classification:} 76B03, 35Q35 }
\renewcommand{\thefootnote}{\arabic{footnote}}

\begin{abstract}
We consider the incompressible 2D Euler equation in an infinite cylinder
 $\mathbb{R}\times \mathbb{T}$
 in the case when the initial vorticity is non-negative, bounded, and compactly supported. We study $d(t)$, the diameter of the support of vorticity, and prove that it allows the following bound:
$d(t)\le
 Ct^{1/3}\log^2t$ when $t\to\infty$.
\end{abstract}\vspace{1cm}
\maketitle
\large

{\Large \section{Introduction}}
Consider the incompressible 2D Euler equation in vorticity form on an infinite
cylinder $S:=\mathbb{R}\times\mathbb{T}$, where $\mathbb{T}=[0,2\pi)$ is a unit circle:
\begin{equation}\label{ecyl}
\partial_t\theta
 + u \cdot \nabla \theta
=
0,
  \quad \theta|_{t=0}=\theta_0.
\end{equation}

The velocity $u(x,y,t)$ is related to
the scalar vorticity $\theta$ via a cylindrical Biot-Savart law, which will be introduced in the next section (see formula \eqref{bscyl}).
This problem is identical to the Euler equation in $\mathbb{R}^2$ with $\theta_0$ being $2\pi$-periodic in $y$ in the sense that we can obtain the two-dimensional cylinder $S$  from the infinite strip
$\mathbb{R}\times[0,2\pi]$ by identifying its sides.
 In the paper, we  use notation $z=(x,y)$, $\xi=(\xi_1,\xi_2)$, and $dz=dxdy, d\xi=d\xi_1d\xi_2$ for shorthand.

We assume that $\theta_0$ has a compact support and $\theta_0(x,y)\in L^\infty(S)$. For the 2D Euler equation on a cylinder, the existence and uniqueness of compactly supported 
solution 
in the sense of distributions
from the class $L^\infty(S)$ can be proved in a similar manner as  in the case of the whole space \cite{yu}. We refer the reader to \cite{ke} or Appendix in \cite{bd}.  If the initial data assumes further $C^{m,\gamma}$-regularity, one can obtain $C^{m,\gamma}$-regular solution  for all time by adapting the method in Chapter 4 of \cite{mb}. In this paper, however, we do not need smoothness that high and from now on a solution   means a  solution of \eqref{ecyl} in the sense of distributions with $u$ given by  \eqref{bscyl}.\\

For any function $f$ compactly supported on $S$, we define
$$d_f:=\sup_{z,\xi\in \it{\rm supp}(f)}|x-\xi_1|,$$
where ${\rm supp}(f)$ denotes the essential support of $f$.

In this paper, we are interested in controlling the support of  nonnegative vorticity for large time. The main result  is the following upper estimate on  $
d_{\theta(t)}$:\smallskip

\begin{theorem}\label{main_theorem} Suppose that an  initial data $\theta_0$
is non-negative, compactly supported, and belongs to $L^\infty(S)$. 
Then, the corresponding solution $\theta$
satisfies 
\begin{equation}\label{main_estimate}
d(t):=d_{\theta(t)}
\le C (t+1)^{\frac 13}\log^2 (2+t)\quad \mbox{for any } t>0\,,
\end{equation} where the constant $C$ 
depends only on $d_{\theta_0}$ and $\|\theta_0\|_{L^\infty}$.\\
\end{theorem}

An important example of $\theta_0$ is the characteristic function $\chi_{\Omega_0}$ of a compact subset $\Omega_0$ of $S$, a patch. Then, $\theta(z,t)=\chi_{\Omega(t)}$ and one can study dynamics of $\Omega(t)$ in time. Note that  
the periodic extension of $\theta_0$ into the whole space $\mathbb{R}^2$ is not compactly supported, in  general. \\

 For the problem when the data is compactly supported in $\mathbb{R}^2$ (so it's not periodic in $y$), the upper bound $d(t)\le C(t+1)^{1/3}$ was obtained in \cite{marchioro}. Later,  it was improved
   to $((t+1)\log (t+2))^{1/4}$ in \cite{isg} (see also \cite{Serfati}).
The key idea of the proof in \cite{isg} was to use the following conserved quantities for Euler equation in $\mathbb{R}^2$:
$$\mbox{the total mass }\int_{\mathbb{R}^2}\theta(z)dz,$$
$$\mbox{the center of mass } \int_{\mathbb{R}^2}z\theta(z)dz,\quad\mbox{and}$$
$$\mbox{the moment of inertia } \int_{\mathbb{R}^2}|z|^2\theta(z)dz\,.$$

In particular, the moment of inertia plays  an important role because its conservation in time shows that,
when the initial vorticity is non-negative and compactly supported near the origin, only a small portion of  $\theta$
can concentrate far away from zero at any given time.  For exterior domains, we refer to \cite{iln, marchioro_ext}.  \\

In order to have an analogous  confinement
 for the 2D Euler on a cylinder $\mathbb{R}\times \mathbb{T}$, one needs to  establish conserved quantities first. It does not seem to be the case that the Euler evolution on a cylinder  preserves the second moment
 $$\int_S x^2\theta(x,y,t)dz.$$
 However, the following quantity:
 $$e_0:=\int_S \theta(z,t) \Psi(z,t) dz\,$$
  is conserved, where
  the stream function $\Psi$ will be introduced in the next section. This allows us to show that
    \begin{equation}\label{intro_unif_bdd}\int_S |x|\theta(x,y,t)dz \end{equation}
   is uniformly bounded in time if the initial vorticity is non-negative (see  Proposition \ref{main_prop}).
 This quantity $e_0$ can be regarded as a regularized energy. In fact, the standard
kinetic energy given by
$$\int_S |u(z,t)|^2 dz$$
 is not finite for non-negative vorticity, in general.  \\

The second ingredient of the proof is related to the cylindrical Biot-Savart law.
It shows that, for the horizontal component of velocity $u_1:=k_1*\theta$, the kernel $k_1$ takes the form $$k_1=
\frac{-\sin(y)}{2(\cosh(x)-\cos(y))}\,.
$$ Thus, it is smaller than  $\frac{C}{|z|}$ near $0$ and decays exponentially for large $|x|$. The decay so strong makes interaction between the distant parts of vorticity essentially negligible. Note that the exponential bound for the kernel $k_1$ has been used in \cite{gs_2} to study
  2D Navier-Stokes equation on a cylinder.  \\

Then, our proof proceeds by controlling the integrals
\[
\int_{x>r}\theta(z,t)dz
\]
for different values of $r\in \mathbb{R}^+$. Using the strong decay of $k_1$, we establish the following   inequality  for $r\gtrsim 1$ (see \eqref{lemma_eq_evol}): \begin{equation*}\int_{|x|>4r}\theta(z,t) dz\lesssim r^{-2}  \int_0^t\left( \left(\int_{|x|>r} \theta(z,\tau)dz\right)^2+ \mbox{ small error}\right)d\tau.   \end{equation*} Then, we analyze the sequence of these estimates  taking $r\sim4^n, n\in \mathbb{N}$ to obtain a bound for $u_1$.   It shows that $u_1$ is very small outside the region  $|x|\gtrsim t^{1/3}\log^2t$ for $t\gtrsim 1$  (see \eqref{main_decay_estimate}) and this will imply  the main estimate \eqref{main_estimate}.\\

For stability questions, 
there were several publications in which stability of steady states on $S$ was studied. In the paper \cite{bm}, the Couette flow was considered. In \cite{cm}, the case of increasing steady vorticity was studied and, more recently, the stability of a rectangular patch was investigated in \cite{bd}.\\

The structure of the paper is as follows. In the next section, the cylindrical Biot-Savart law  and some conserved quantities
of the Euler equation will be introduced. In Section 3, we will prove the bound for
\eqref{intro_unif_bdd}
 and discuss some of its easy consequences. Section 4 contains the proof of
  Theorem~\ref{main_theorem}. We collect auxiliary results in the last section.\bigskip

In this paper, we  use the following standard notation. If two non-negative functions $f_1$ and $f_2$ satisfy $f_1\le Cf_2$ with some absolute constant $C$, we write $f_1\lesssim f_2$. If $f_1\lesssim f_2$ and $f_2\lesssim f_1$,  we use $f_1\sim f_2$. If $f_{1(2)}(t)$ satisfy $f_1(t)\le Cf_2(t)$ for all $t>1$, we write $f_1=\mathcal{O}(f_2)$ as $t\to\infty$, here $C$ might depend on some fixed parameters but not on $t$. As usual, $\mathbb{N}$ denotes the set of natural numbers,  $\mathbb{Z}^+:=\mathbb{N}\cup \{0\}$.

{\Large \section{Preliminaries}}

   In this paper, we use the cylindrical Biot-Savart law (see
\cite{am} or \cite{gs, bd} for the detail):
\begin{equation}\label{bscyl}
u(x,y)=(u_1,u_2)= k*\theta=\int_S k(x-\xi_1,y-\xi_2)\theta(\xi_1,\xi_2)d\xi,\end{equation}
   where
the kernel $k$ is given by 
\begin{equation*}
k(x,y) 
=(k_1,k_2)=
\frac{(-\sin(y),\sinh(x))}{2(\cosh(x)-\cos(y))}.
\end{equation*}

Note that  if we define the stream function $\Psi$ of $\theta$ by
$\Psi(x,y)=\Gamma*\theta
$, where
\begin{equation*}
\Gamma(x,y) = \frac{1}{2}\log
(\cosh(x)-\cos(y)),
\end{equation*} then  the   function  $\Psi$
 solves the elliptic
problem
\begin{equation*}\label{strcyl}
(2\pi)^{-1} \Delta \Psi = \theta,\quad \lim_{x\rightarrow
+\infty}\partial_x \Psi(x,y) = - \lim_{x\rightarrow
-\infty}\partial_x \Psi(x,y), \quad |\Psi(x,y)|\le C(|x|+1)
\end{equation*} and the velocity can be recovered by
$u =\nabla^\perp\Psi=(-\partial_y\Psi,\partial_x\Psi)$. 
We observe that $|\Gamma(z)|\sim |\log|z||$ for small $|z|$
and $\Gamma(z)\sim |x|$ for large $|x|$.\\
\\


For any bounded and compactly supported $\theta_0$, we denote
\[
\mbox{the total mass }  m_0:=\int_S \theta_0(z)dz,
\]
\[
\mbox{the horizontal center of  mass }  h_0:=\int_S x\theta_0(z)dz,\quad \mbox{ and}
\]
 \[
\mbox{the regularized energy }  e_0:=\int_S \theta_0(z)\Psi_0(z)dz=\int_{S}\int_{S} \theta_0(z) \theta_0(\xi) \Gamma(z-\xi)d\xi dz.
\]

{\bf Remark.}\label{dependency_remark}
Both  $m_0$ and $e_0$ are controlled by diameter $d_{\theta_0}$ and  $\|\theta_0\|_{L^\infty{(S)}}$.
Indeed, we have
\begin{equation}\label{sot1}
|m_0|\le \|\theta_0\|_{L^1{(S)}}  \lesssim d_{\theta_0} \cdot  \|\theta_0\|_{L^\infty{(S)}}.
\end{equation}
For the regularized energy $e_0$, we notice that  there is $l\in \mathbb{R}$, such that $\theta_0$ is supported in the rectangle $\{z\in S\, |\, |x-l|\leq d_{\theta_0}\}$. For any such $z$, the stream function
$\Psi_0$ of $\theta_0$ satisfies $$|\Psi_0(z)|\le \|\theta_0\|_{L^\infty{(S)}}\cdot\int_{\{\xi\in S\,|\,  |\xi_1-l|\leq d_{\theta_0}\}} | \Gamma(z-\xi)| d\xi\le \|\theta_0\|_{L^\infty{(S)}}\cdot  C_{d_{\theta_0}}$$ since $|z-\xi|\lesssim 1+d_{\theta_0}$ and $\Gamma(\cdot)$ is locally integrable thanks to \[|\Gamma(z)|\sim |\log|z||\] which holds for small $|z|$. 
 So we have
\begin{equation} \label{sot2}
|e_0|\le   \|\theta_0\|_{L^1{(S)}}\cdot  \|\theta_0\|_{L^\infty{(S)}}\cdot  C_{d_{\theta_0}} \lesssim d_{\theta_0} \cdot  \|\theta_0\|^2_{L^\infty{(S)}}\cdot C_{d_{\theta_0}}. \end{equation}

{\bf Remark.}
The kinetic energy
$\int_S |u|^2dz$ is not finite, in general. Indeed,  assume  that the data $\theta_0$ is non-negative and non-trivial. Since
$k_2\rightarrow \pm \frac{1}{2}$ as $x\rightarrow \pm\infty $, we get
$|u_2|=|k_2*\theta|\rightarrow  \frac{1}{2}m_0\neq 0$ as $x\rightarrow \pm\infty$. This implies divergence of the integral $\int_S |u|^2dz$.\\

Since we consider an incompressible flow and its vorticity is transported by the flow, the $L^1-$norm and $L^\infty-$norm of $\theta(z,t)$ are preserved in time. In addition to these norms, we have the following conserved quantities. \smallskip

\begin{lemma}\label{simple_lemma}
For any bounded and compactly supported $\theta_0$,
the Euler evolution on $S$ preserves
the total mass, the horizontal center of mass, and the regularized energy:
\[m_0=\int_S \theta(z,t)dz, \quad h_0=\int_S x\theta(z,t)dz, \quad e_0=\int_{S\times S} \theta(z,t) \theta(\xi,t) \Gamma(z-\xi)d\xi dz\quad\mbox{ for all } t\ge 0.\]

\end{lemma}
Its proof can be found in Proposition 2.1 of \cite{bd}. \smallskip

{\bf Remark. } If $\theta$ is a smooth solution, this lemma easily follows from the following arguments. The quantity \[\int_S \theta(z,t)dz\] is time-independent because  the velocity $u$ is incompressible. To handle the center of mass, we multiply  equation \eqref{ecyl} by $x$ and integrate over $S$ to get
\[
\frac{d}{dt}
 \int_{S} x\theta(x,y,t)dz=- \int_{S} x(-\Psi_y\theta_x+\Psi_x\theta_y)dz\,,
\]
where $\theta$ is smooth and compactly supported. Integration by parts gives
\[
- \int_{S} x(-\Psi_y\theta_x+\Psi_x\theta_y)dz
=-\int_{S} \Psi_y\theta dz=\int_{S\times S} \frac{-\sin(y-\xi_2)}{2(\cosh(x-\xi_1)-\cos(y-\xi_2))}\theta(z,t)\theta(\xi,t)dzd\xi=0\,,
\]
because the kernel in this quadratic form is antisymmetric. Consider the regularized energy. Differentiation in time gives
\begin{equation*}\begin{split}
\frac{d}{dt}&
\left(\int_{S\times S} \theta(z,t) \theta(\xi,t) \Gamma(z-\xi) d\xi dz\right)=
2\int_{S\times S} \theta_t(z,t)\theta(\xi,t)\Gamma(z-\xi)dzd\xi
\\
&\hspace{2cm}=2\int_S (\Psi_y\theta_x-\Psi_x\theta_y)\Psi dz=\int_S (\Psi^2)_y\theta_x-(\Psi^2)_x\theta_ydz=0
\end{split}\end{equation*}
after integration by parts. For  solutions in the sense of distributions,  a mollification argument has been used in \cite{bd}.

{\bf Remark.}
For any bounded and compactly supported initial vorticity, we have   a trivial bound $$ 
d(t)
=\mathcal{O}(t)\quad  \mbox{ as   }t \to\infty.$$ Indeed, since the vorticity in the Euler equation is transported by the corresponding velocity and our domain is a horizontal cylinder, 
 we only need to estimate the horizontal velocity $u_1=k_1*\theta$,
where $$k_1(z)=\frac{-\sin(y)}{2(\cosh(x)-\cos(y))}.$$
We can use an estimate
$|k_1(z)|\lesssim |z|^{-1}e^{-|x|/2} $  to get the bound $|u_1|\lesssim
\|\theta\|_{L^1{(S)}}+\|\theta\|_{L^\infty{(S)}}
$. 
 Since $L^p-$norms of $\theta$ are preserved by the Euler evolution,
    we get
  \begin{equation}\label{bdd_speed}
  \sup_{z\in S,t\ge 0}|u_1(z,t)|\lesssim
\|\theta_0\|_{L^1{(S)}}+\|\theta_0\|_{L^\infty{(S)}}\lesssim
  d_{\theta_0}\cdot \|\theta_0\|_{L^\infty{(S)}} + \|\theta_0\|_{L^\infty{(S)}},
  \end{equation}  which implies $d(t)=\mathcal{O}(    t)$.\bigskip

{\Large \section{One proposition and another rough bound on $d(t)$}}
In the following proposition,  the non-negativity of $\theta$ will be crucial.
\begin{proposition}\label{main_prop} Suppose that $\theta_0$ is non-negative, bounded, and compactly supported on $S$.
If the horizontal center of mass is at $0$,
then
\begin{equation*}
\sup_{t\ge 0} \int_S |x|\theta(z,t)dz\le  C(d_{\theta_0},\|\theta_0\|_{L^\infty{(S)}}) \,.
\end{equation*}
\end{proposition}
\begin{proof}
Bounds \eqref{sot1} and \eqref{sot2} show that it is enough to estimate
$\int_S |x|\theta(z,t)dz$  by $m_0$, $e_0$, and $\|\theta_0\|_{L^\infty{(S)}}$.
Let $t\ge 0$. Then, by Lemma \ref{simple_lemma},
 we write
\[
e_0= \int_S\int_S \theta(z,t)\theta(\xi,t)\Gamma(z-\xi)d\xi dz=\sum_{n\in \mathbb{Z}}\sum_{m\in \mathbb{Z}} \int_{n<x<n+1}\int_{m<\xi_1<m+1} \theta(z,t)\theta(\xi,t)\Gamma(z-\xi)d\xi dz.
\]
For each term in the sum which satisfies $|n-m|<10$, we can write
\[
 \left|\int_{n<x<n+1}\int_{m<\xi_1<m+1} \theta(z,t)\theta(\xi,t)\Gamma(z-\xi)d\xi dz\right|\lesssim \|\theta(t)\|_{L^\infty{(S)}}\cdot  \int_{n<x<n+1}\theta(z,t) dz\,.
\]
Indeed, it follows from the logarithmic estimate for the kernel \[
|\Gamma(\eta)|
\sim |\log|\eta||\] which holds for small $|\eta|$. Thus, the sum of all terms for which  $|n-m|<10$ is  bounded  by $C   \|\theta(t)\|_{L^\infty{(S)}}\cdot \int_S\theta(z,t)dz=C \|\theta_0\|_{L^\infty{(S)}}\cdot m_0$. \\

On the other hand, since $\Gamma(\eta)\sim  |\eta_1|$ for $|\eta_1|\ge 2$, all terms for which $|n-m|\ge 10$ satisfy
\[
 \int_{n<x<n+1}\int_{m<\xi_1<m+1} \theta(z,t)\theta(\xi,t)\Gamma(z-\xi)d\xi dz \sim \int_{n<x<n+1}\int_{m<\xi_1<m+1} \theta(z,t)\theta(\xi,t) |x-\xi_1|d\xi dz
\]
and, in particular,  they  are all positive. Suppose that $\theta_0$ is non-trivial so that $\int_S \theta(z,t) dz=m_0>0$. Then, we have  $\int_{x>0}\theta (z,t) dz \ge m_0/2$, or  $\int_{x<0}\theta (z,t)dz\ge m_0/2$, or the both estimates are true. Suppose, e.g.,
\[
 \int_{x>0}\theta(z,t) dz \ge m_0/2.
\]
 Then, we can write
 \begin{equation*}
 \begin{split}
  \frac{m_0}{2}\int_{x<-10}|x|\theta(z,t)dz &\le \int_{x<-10}\int_{\xi_1>0}\theta(z,t)\theta(\xi,t)|x| d\xi dz\\
&  \le \int_{x<-10}\int_{\xi_1>0}\theta(z,t)\theta(\xi,t)|x-\xi_1|d\xi dz \lesssim |e_0|+m_0\cdot \|\theta_0\|_{L^\infty{(S)}}.
 \end{split}
 \end{equation*}
Thus, we get
\[
\int_{x<0}|x|\theta(z,t)dz\lesssim
\frac{1}{m_0}\cdot(|e_0|+m_0\cdot \|\theta_0\|_{L^\infty{(S)}})+ \|\theta_0\|_{L^\infty{(S)}}\,.
\]
Since the horizontal center of mass is at zero for all time by Lemma \ref{simple_lemma}, we have
\[
\int_{x>0}x\theta(z,t)dz=\int_{x<0}|x|\theta(z,t)dz
\] and, therefore,
\[
\int_S |x|\theta(z,t) dz\le C(e_0,m_0, \|\theta_0\|_{L^\infty{(S)}})\,,
\] where the constant is independent of $t\ge 0$.


\end{proof}

Before proving Theorem \ref{main_theorem},
we will show how Proposition \ref{main_prop} can be used to obtain a rough upper estimate on the diameter. Under its assumptions, we get
\begin{equation}\label{sot3}
 \int_{r\le  |x| }\theta(z,t) dz 
\le  \frac{1}{r}\int_{r\le  |x| }{|x|}\theta(z,t) dz
\le   \frac{ C(d_{\theta_0},\|\theta_0\|_{L^\infty{(S)}})     }{r}\quad \mbox{ for any } r>0 \mbox{ for any } t\ge0.
\end{equation}
Then,  we have an estimate for the first component of the velocity:
\begin{equation}\label{sot4}
|u_1(z,t)|\lesssim \int_S \theta(\xi,t) \left|\frac{\sin(y-\xi_2)}{\cosh(x-\xi_1)-\cos(y-\xi_2)}\right|d\xi\lesssim
\int_{|x-\xi_1|\ge 1} {\theta(\xi,t)}{e^{-|x-\xi_1|}}d\xi+\int_{|x-\xi_1|<1} \frac{\theta(\xi,t)}{|z-\xi|}d\xi.
\end{equation}
The first integral can be estimated as
\begin{equation*}\begin{split}
&\hspace{2cm} \int_{|x-\xi_1|\ge 1} {\theta(\xi,t)}{e^{-|x-\xi_1|}}d\xi=
\sum_{n\in \mathbb{Z}} \int_{|x-\xi_1|>1, n<\xi_1<n+1}\theta(\xi,t) e^{-|x-\xi_1|}d\xi \\
&\lesssim
\sum_{n\in \mathbb{Z}} e^{-|x-n|}\int_{ n<\xi_1<n+1}\theta(\xi,t)d\xi\le  C(d_{\theta_0},\|\theta_0\|_{L^\infty{(S)}})
\sum_{n\in \mathbb{Z}} e^{-|x-n|}(|n|+1)^{-1}\le \frac{C(d_{\theta_0},\|\theta_0\|_{L^\infty{(S)}})}{|x|+1}\,.
\end{split}\end{equation*}
Here we used \eqref{sot3} to bound $\int_{n<\xi_1<n+1}\theta d\xi$ for $n\neq \{0,-1\}$. For  $n=\{0,-1\}$, we wrote $\int_{n<\xi_1<n+1}\theta d\xi\lesssim \|\theta\|_{L^\infty(S)}$.

The second integral in \eqref{sot4} can be bounded by H\"older
inequality as follows
\[
\int_{|x-\xi_1|<1} \frac{\theta(\xi,t)}{|z-\xi|}d\xi\le \||z-\xi|^{-1}\|_{L^{p}(|x-\xi_1|<1)}\|\theta\|_{L^{p'}(|x-\xi_1|<1)}\lesssim C(\epsilon) \left( \int_{|x-\xi_1|<1}\theta^{p'}d\xi \right)^{\frac {1}{p'}}\,,
\]
where $p=2-\epsilon, \epsilon\in (0,1)$, and $p'$ is defined by $p^{-1}+p'^{-1}=1$. Finally, we have
\[
 \int_{|x-\xi_1|<1}\theta^{p'}d\xi\lesssim \|\theta\|_{L^\infty(S)}^{p'-1} \int_{|x-\xi_1|<1}\theta d\xi\lesssim \frac{C(\epsilon,d_{\theta_0},\|\theta_0\|_{L^\infty{(S)}}) }{|x|+1}\,.
\]
Thus, for the first component in the Lagrangian dynamics of a point $(x(t),y(t))$ in the support of $\theta(z,t)$, we  have
\[
|\dot{x}|\le C(\delta,d_{\theta_0},\|\theta_0\|_{L^\infty{(S)}})  (|x|+1)^{-\frac 12+\delta}
\quad\mbox{ for } t>0\,
\]
with arbitrary $\delta>0$.
This gives the bound \begin{equation}\label{easy_con}d(t)\le  C(\delta_1,d_{\theta_0},\|\theta_0\|_{L^\infty})(1+t)^{\frac 23+\delta_1} \quad\mbox{ for \, every } \delta_1>0.
\end{equation} \bigskip

{\Large \section{
Proof of Theorem \ref{main_theorem}
}}
In this section, we show that the bound \eqref{easy_con} can be improved to
 $d(t)=\mathcal{O}(t^{
1/3}\log^2 t)$. To do that we will exploit the decay of $u_1$ both in $z$ and $t$.\\

\begin{proof}[Proof of Theorem \ref{main_theorem}]

Without loss of generality, we suppose that
 $0\le\theta_0\le 1$, $d_{\theta_0}\le 1$, $\int_S  \theta_0(z)dz>0$, and $\int_S x \theta_0(z)dz=0$. 
Then, $\theta_0$ is supported in
$[-1,1]\times \mathbb{T}$ and $\int_S  \theta_0(z)dz\le 2\pi$. Indeed, notice first that if $\theta(z,t)$ is the solution to \eqref{ecyl} and if $\widehat x\in\mathbb{R}$, then $\theta(x-\widehat x,y,t)$ is the solution to \eqref{ecyl} with the initial data $\theta_0(x-\widehat x,y)$. Thus, choosing $\widehat x$ suitably, we can always assume that $h_0=0$. Then, if  $\theta_0$ is an arbitrary non-trivial, non-negative, bounded, and compactly supported function satisfying $\int_S x \theta_0(z)dz=0$ and giving rise to the solution $\theta$,
we can rescale  $\theta$ as follows.
Put $M: =\|\theta_0\|_{L^\infty}>0$ and  choose $   N\in \mathbb{N}$ large enough so that  $ N\ge d(0)$. Then, define $\widetilde{\theta}$ by
 $$\widetilde{\theta}(x,y,t):=\frac{1}{M}\theta(N\cdot x,N\cdot y,\frac{t}{M}).$$
Notice now that  $\widetilde{\theta}$ is $2\pi/N$-periodic in $y$,  solves the Euler equation \eqref{ecyl}, and satisfies
\[0\le\widetilde{\theta}_0\le 1, \, d_{\widetilde{\theta}_0}\le 1,\, \int_S \widetilde{\theta}_0(z)dz>0,\,\int_S x \widetilde{\theta}_0(z)dz=0.\]

We will need the following  lemma.

\begin{lemma}\label{euler_evol_lem}
Take $a\in 2\mathbb{N}$.  We have
 \begin{equation}\label{lemma_eq_evol}
 \int_{|x|>2a}\theta(z,t) dz\lesssim a^{-2}\int_0^t \left(\left(\int_{|x|>a/2} \theta(z,\tau)dz\right)^2+e^{-a/4}  \left(\int_{|x|>a/2} \theta(z,\tau)dz\right) \right)d\tau
 \end{equation}
 for any $t\ge 0$.
\end{lemma}

\begin{proof}
 For $a\in 2\mathbb{N}$, consider
\[
k_a(t):=\int_{x>a} (x-a)^2\theta(z,t)dz.
\]
If $\theta$ is 
smooth,
taking the time derivative of $k_a$ gives
\begin{equation*}\begin{split}
 k_a'(t)&=
\int_{x>a} (x-a)^2(\partial_t\theta) dz =
-\int_{x>a} (x-a)^2(u\cdot\nabla\theta) dz\\
&=\int_{x>a} (x-a)^2(\Psi_y\theta_x
-\Psi_x \theta_y)dz
= 2\int_{x>a} (x-a)\cdot\theta\cdot(-\Psi_y) dz.
\end{split}\end{equation*}
Recall $-\Psi_y=u_1=k_1*\theta$.
We estimate the time derivative:
\begin{equation*}\begin{split}
&|k_a'(t)|=\left|2\int_{x>a} (x-a)\cdot \theta \cdot u_1 dz\right|\lesssim \left| \int_{x>a}\int_S \frac{(x-a)\sin(y-\xi_2)}{\cosh(x-\xi_1)-\cos(y-\xi_2)} \theta(z,t)\theta(\xi,t)d\xi dz \right|\\
&\le   \left| \int_{x>a}\int_{\xi_1<a}
\dots  \,d\xi dz 
\right|
+
\left| \int_{x>a}\int_{\xi_1>a}
\dots  \,d\xi dz 
\right|.
\end{split}\end{equation*}
Using the bound $|x-a|\le|x-\xi_1|$ in the first term and symmetrizing in the second one,
we get
\[
|k_a'(t)|
\lesssim  \int_{x>a}\int_S \left|\frac{(x-\xi_1)\sin(y-\xi_2)}{\cosh(x-\xi_1)-\cos(y-\xi_2)} \right| \theta(z,t)\theta(\xi,t)
 d\xi dz. \]
We observe that 
\[
 \left|\frac{(x-\xi_1)\sin(y-\xi_2)}{\cosh(x-\xi_1)-\cos(y-\xi_2)} \right|\lesssim (1+|x-\xi_1|)e^{-|x-\xi_1|}\lesssim
 e^{-\frac{1}{2}|x-\xi_1|}\,.
\]
Let us  denote   $\int_{a\le \xi<b}\theta(z,t)dz$  by $\int_a^b\theta$ for shorthand. The above estimate implies
\[
|k_a'|
\lesssim \sum_{j=-\infty}^\infty\sum_{l=a}^\infty e^{-{\frac{1}{2}}|j-l|}\left(\int_{j}^{j+1}\theta\right)\left( \int_{l}^{l+1}\theta\right).
\]
We get
\[
 \sum_{j=a}^\infty\sum_{l=a}^\infty e^{-{\frac{1}{2}}|j-l|}\left(\int_{j}^{j+1}\theta\right)\left( \int_{l}^{l+1}\theta\right)
 \le
  \left(\sum_{j=a}^\infty \int_j^{j+1}\theta\right)^2= \left(\int_a^\infty \theta\right)^2 
 \] and
  \begin{equation*}
 \begin{split}
 & \sum_{j=-\infty}^a \sum_{l=a}^\infty e^{-{\frac{1}{2}}|j-l|}\left(\int_{j}^{j+1}\theta\right)\left( \int_{l}^{l+1}\theta\right)\\
  &
\le
    \sum_{j=a/2}^a \sum_{l=a}^\infty e^{-{\frac{1}{2}}|j-l|}\left(\int_{j}^{j+1}\theta\right)\left( \int_{l}^{l+1}\theta\right)+
   \sum_{j=-\infty}^{a/2} \sum_{l=a}^\infty e^{-{\frac{1}{2}}|j-l|}\left(\int_{j}^{j+1}\theta\right)\left( \int_{l}^{l+1}\theta\right)\\
   &\lesssim\left(\int_{a/2}^\infty \theta\right)^2+e^{-a/4}\left(\int_{a/2}^\infty \theta\right).
 \end{split}
 \end{equation*}
Since $\int_a^\infty (x-a)^2\theta_0=0$, we can write
 \[
 \int_a^\infty  (x-a)^2\theta \lesssim \int_0^t \left(\left(\int_{a/2}^\infty \theta\right)^2+e^{-a/4}  \left(\int_{a/2}^\infty \theta\right) \right)d\tau
 \]
 and
 \[
 \int_{2a}^\infty\theta\lesssim a^{-2}\int_0^t \left(\left(\int_{a/2}^\infty \theta\right)^2+e^{-a/4}  \left(\int_{a/2}^\infty \theta\right) \right)d\tau.
 \] 
The estimate for $\int_{-\infty}^{-2a}\theta$  can be proved similarly. Thus, we get 
  \eqref{lemma_eq_evol} for smooth solutions. 
For  solutions in the sense of distributions, one can use the mollification argument following, e.g., \cite{bd}, Proposition 2.1.

 \end{proof}

We denote
\[  f_{0}(t):=\int_S \theta(z,t)dz  \quad  \mbox{ and }  \quad
f_n(t):=\int_{|x|>4^n} \theta(z,t)dz,\quad n \in \mathbb{N}.
\]
So,
$f_{0}(t)=f_{0}(0)= \int_S  \theta_0(z)dz\lesssim 1$ for $t\ge0$ and
 $f_n(0)=0$ for $n\in \mathbb{N}$.
 Taking $a=2 \cdot 4^n$ in Lemma~\ref{euler_evol_lem}, 
we get the bounds
\[
 f_{n+1}(t)\lesssim  4^{-2n}\int_0^t (f_n^2(\tau)+e^{-{\frac{1}{2}}4^n} f_n(\tau))d\tau \quad\mbox{ for any } n\ge 0, t>0
\]
and Proposition \ref{main_prop} yields  $f_n(t)\lesssim 4^{-n}  \mbox{ for any } n\ge 0, t\ge 0$. We combine these two estimates into
 \begin{equation*}
f_0(t)=c_1,  f_{n+1}(t)\le c_2 \min\Big( 4^{-2n}\int_0^t (f_n^2(\tau)+e^{-{\frac{1}{2}}4^n} f_n(\tau))d\tau,      4^{-(n+1)}\Big)   \quad\mbox{for any } n\ge 0, t\ge 0\,,
 \end{equation*}
where time-independent parameters $c_1$ and $c_2$ satisfy $0<c_1\lesssim 1,0<c_2\lesssim 1$.\smallskip

For any bounded and non-negative function $h$ and for $n\ge0$, we define the  operator $M_n$ by
\begin{equation*}
(M_n h)(t)=c_2\min\left( 4^{-2n}\int_0^t h(\tau)\cdot (h(\tau)+e^{-{\frac{1}{2}}4^n})d\tau, 4^{-(n+1)}  \right) \quad \mbox{for } t\ge0.
\end{equation*}
Without loss of generality, we can assume $c_2\ge 2$. Define $\{g_n(t)\}$ recursively by
\begin{equation*}
g_0(t)=c_1, g_{n+1}:=M_n(g_n)\
\end{equation*}
for all $t\ge 0$. We can use induction argument to show that
\begin{equation}
\label{main_lemma_technical1}
   f_{j}(t)\le g_{j}(t)
  \quad\mbox{ for any }   j\ge 0, \,t\ge 0\,.
\end{equation}
From Lemma~\ref{main_lemma},  proved in the next section, we know that there exist $n_0\in \mathbb{N}$ and positive constants $c_3,c_4,c_5$ such that
\begin{equation}\label{sera1}
g_{n+j}(t) \le c_3{4^{-n-c_42^j}}\,,
\end{equation}
for all $n\ge n_0$, $j\in \mathbb{Z}^+$, and $t\in [0,c_54^{3n}]$.


We now can estimate the first component of velocity.
By  \eqref{main_lemma_technical1} and  \eqref{sera1}, 
 we have
  \begin{equation}\label{cor1}
   \int_{|x|\ge 4^{n+j}}\theta(z,t)dz \le c_3{4^{-n-c_42^j}}
   \end{equation}
 for any $n\ge n_0$, $j\ge 0$, and $ 0\le t\le  c_5 4^{3n}$.   If necessary, redefine $n_0$ to satisfy $e\le  c_5 4^{3(n_0-1)}$ 
     and let $${T}:= c_5 4^{3(n_0-1)}\ge e.$$ 
     Now, given any $t\ge T$, we choose $n$ to depend on $t$ in such a way that
   $ c_5 4^{3(n-1)}\le  t<  c_5 4^{3n}$.  Substituting this bound into \eqref{cor1}, we get  
  \begin{equation*}
   \int_{|x|\ge 4\left({\frac{t}{c_5}}\right)^{1/3}4^{j}}\theta(z,t)dz \le c_3\left({\frac{c_5}{t}}\right)^{\frac 13}{4^{-c_42^j}}.
   \end{equation*}
   for all $t\ge T$ and $j\ge 0$.\\

   For $A>1/2$ and for $t\ge {T}$, we can  take the integer $j=j(A,t)\ge 0$ such that $2^{j-1}<A\log t\le  2^j$. Then,
    \[
   \int_{|x|\ge 16A^2\left({\frac{t}{c_5}}\right)^{1/3}{
\log^2 t
   }}\theta(z,t)dz \le (c_3\cdot c_5^{\frac 13 }){t^{-(\frac 13+(c_4\log 4) A)}}.
   \]
 Introducing
 \[
 \phi(L):=16\left(\frac{L-\frac 13}{c_4\log 4}\right)^2/c_5^{\frac 13},\quad
 c_{7}:= c_3\cdot c_5^{\frac 13 }, \quad L:=\frac 13+(c_4\log 4)A\,,
 \]
  and assuming that $L>L_0:=(1/3)+(c_4\log 4)/2$, 
   we can rewrite the last inequality in more convenient form
    \begin{equation}\label{main_decay_estimate}
   \int_{|x|\ge \phi(L)t^{1/3}\log^2 t
   }\theta(z,t)dz \le c_{7}{t^{-L}}
   \end{equation}
   for all $t\ge T$. Note  that $\phi(L)\sim L^2$ for $L> L_0$. \\


For $L>L_0$, we define $$R_L(t):=
2(\phi(L)t^{1/3}\log^2 t+1)$$ for $0\le  t<\infty$. We need the following lemma.

\begin{lemma} \label{sort1}There exists  $L_1>0$ such that $$|u_1(z,t)|\le  \frac{d}{dt}R_L(t)
 $$ holds whenever $L\ge L_1, |x|=R_L(t)$, and $t\ge {T}$.
\end{lemma}

\begin{proof}
Let $L> L_0$. We have a bound \begin{equation*}\begin{split}
|u_1(z,t)|&\lesssim  \int_{|x-\xi_1|\ge 1} \theta(\xi,t)  e^{-|x-\xi_1|} d\xi +\int_{|x-\xi_1|<1} \frac{\theta(\xi,t)}{|z-\xi|}d\xi\\
\end{split}\end{equation*} for any $z$ and for any $t$. Notice that   $\phi(L)t^{1/3}\log^2 t\le  \min(R_L(t)/2,R_L(t)-1)$. Suppose $x=R_L(t)$, the case $x=-R_L(t)$ can be handled similarly. Thus,  for $t\ge {T}$, we get
\begin{equation*}\begin{split}
&|u_1(R(t),y,t)|\lesssim  \\
&
\int_{\xi_1<R_L(t)/2} \theta(\xi,t)  e^{-|R_L(t)-\xi_1|}d\xi+ \int_{\xi_1\ge R_L(t)/2} \theta(\xi,t)  e^{-|R_L(t)-\xi_1|} d\xi+ \int_{|R_L(t)-\xi_1|<1} \frac{\theta(\xi,t)}{|z-\xi|}d\xi \\
&\lesssim e^{-R_L(t)/2}\cdot\int_{\xi_1<R_L(t)/2} \theta(\xi,t)d\xi  + \int_{\xi_1\ge R_L(t)/2} \theta(\xi,t)d\xi +
\Big(\int_{|R_L(t)-\xi_1|<1} |{\theta(\xi,t)}|^3d\xi\Big)^{\frac{1}{3}}\,,
\\
\end{split}\end{equation*}
where we used H\"older inequality to get the  last term. Recall that
$\|\theta\|_{L^\infty(S)}\le 1, \|\theta\|_{L^1(S)}\lesssim 1$, so
\begin{equation*}\begin{split}
&|u_1(R_L(t),y,t)| \lesssim \\
&    e^{-R_L(t)/2}\cdot\int_S\theta(\xi,t) d\xi + \int_{\xi_1\ge R_L(t)/2} \theta(\xi,t)d\xi +
\|\theta\|_{L^\infty}^{\frac{2}{3}}\Big(\int_{\xi_1>R_L(t)-1 } |{\theta(\xi,t)}|d\xi\Big)^{1/3}\\
&\lesssim  e^{-(\phi(L)t^{1/3}\log^2 t+1)}  + \int_{\xi_1\ge \phi(L)t^{1/3}\log^2 t} \theta(\xi,t)d\xi +
\Big(\int_{\xi_1\ge \phi(L)t^{1/3}\log^2 t} |{\theta(\xi,t)}|d\xi\Big)^{1/3}\,.
\end{split}\end{equation*}
We now use  \eqref{main_decay_estimate}  to   get
\[
|u_1(R_L(t),y,t)|\le c_8\Bigl(  e^{-\phi(L)t^{1/3}\log^2 t}  + {t^{-L}}+
t^{-\frac L3}\Bigr)
\]
with some constant $c_8$.

The derivative of $R_L(t)$ can be computed explicitly: \[
\frac{d}{dt}R_L(t)=2\phi(L)t^{-2/3}\log t \left(\frac{1}{3}\log t +2\right).\] Recalling that $\phi(L)\sim L^2$ when $L\to\infty$, we can take  $L_1$ large enough  to have
$$ c_8\Bigl( e^{-\phi(L)t^{1/3}\log^2 t}  + {t^{-L}}+
{t^{-L/3}}\Bigr)\le  \frac{d}{dt}R_L(t) $$
uniformly in $t\ge T$ and $L\ge L_1$. \\

\end{proof}
We are ready to finish the proof of Theorem \ref{main_theorem}. We claim that there is  an absolute constant $L$ such that
\begin{equation}\label{sog1}
d_{\theta(t)}\le 2R_L(t)
\end{equation}
for all $t\ge 0$.

Indeed, notice first that  $u_{max}:=\sup_{t\ge 0}\|u_1(t)\|_{L^\infty}<\infty$ by
\eqref{bdd_speed}. Since $\phi(L)\sim L^2$, there is $L_2>0$ so that for each  $L\ge L_2$ we get
 $$R_L(t)\ge 1+  u_{max}\cdot t $$  uniformly in  $t\in [0,{T}]$. Since $\theta_0$ is supported in $[-1,1]\times \mathbb{T}$, the Euler solution  $\theta(z,t)$  is supported in $[-R_L(t),R_L(t)]\times \mathbb{T}$ for all $t\in [0,{T}]$ and  for all $L\ge L_2$.\\ 

Take $L\ge \max\{L_0,L_1,L_2\}$. From Lemma \ref{sort1}, we conclude that
for any particle trajectory $Z_{{(x,y)}}(t)=(X_{{(x,y)}}(t),Y_{{(x,y)}}(t))$ in Lagrangian dynamics, satisfying
\begin{equation*}\begin{cases} \frac{d}{dt}Z_{{(x,y)}}(t)&=u(Z_{{(x,y)}}(t),t)\quad \mbox{ for } t>{T}\\
Z_{{(x,y)}}({T})&= {(x,y)}
\in [-R_L({T}),R_L({T})]\times\mathbb{T}
\end{cases}, \end{equation*} we have
$Z_{{(x,y)}}(t)\in [-R_L(t),R_L(t)]\times\mathbb{T}$ for any $t\ge {T}$. Indeed, we argue by contradiction:  if there is a  particle trajectory escaping from the region, then there should be a moment $T_0\geq T$
and a point $z_0=(x_0,y_0)\in S$  such that $|x_0|=  R_L(T_0)$ and $|u_1(z_0,T_0)|>\frac{d}{dt}R_L(t_0)$, which contradicts the lemma.

Estimate \eqref{sog1} finishes the proof of Theorem. \\

\end{proof}





\section{Some auxiliary results}

Recall that the operator $M_n$ has been defined as
\[
(M_n h)(t)=c_2\min\left( 4^{-2n}\int_0^t h(\tau)\cdot (h(\tau)+e^{-{\frac{1}{2}}4^n})d\tau, 4^{-(n+1)}  \right) \quad \mbox{for } t\ge0
\]
and $c_2\ge 2$. Take $c_1>0$ and define $\{g_n(t)\}$ recursively by
\begin{equation}\label{sov3}
g_0(t)=c_1, g_{n+1}:=M_n(g_n)\quad \mbox{for all } t\ge 0.
\end{equation}

 \begin{lemma}\label{main_lemma}
 There exists an integer $n_0\in\mathbb{N}$ and positive constants $c_3, c_4,c_5$ such that\\ for any $n\ge n_0$ and  for $ 0<t\le  c_5 4^{3n}$, we have
\begin{equation}\label{main_lemma_technical}
    g_{n+j}(t) \le c_3{4^{-n-c_42^j}}
  \quad\mbox{ for any }   j\ge 0.
\end{equation}

 \end{lemma}

\begin{proof}
  Without loss of generality, we assume $c_2\ge 2$.
For any bounded and non-negative function $h$ and for $n\ge0$, $M_nh(\cdot)$ is non-decreasing in $t$. We denote by $T_n(h)$  the first time when $M_nh(t)=c_24^{-n-1}$. If $h$ is non-decreasing and  if $h$ is not identically zero,  we have
$0<T_n(h)<\infty$. 
Moreover,
\begin{equation*}\label{sde}
T_n(h_1)\le T_n(h_2)
\end{equation*}
if $h_1\ge h_2\ge 0$ for all $t$. Function $g_n$ defined in \eqref{sov3} satisfies the following properties.

$\bullet$\quad For all $n\in \mathbb{Z}^+$, $g_n$ is non-decreasing, bounded, non-negative, and $g_n(t)>0$ for any $t>0$.

$\bullet$\quad Denote $t_n:=T_{n-1}(g_{n-1})<\infty$ for $n\ge1$ and $t_0:=0$. In other words, $t_n=\min \{\tau : g_n(\tau)=c_24^{-n}\}$ for $n\ge1$ . We will need some estimates on  $t_n$ later on so we start with getting a lower bound.

 Since $e^{-\alpha}<1/ \alpha$ for $\alpha>0$ and $c_2\ge 2$, we get
$e^{-(\frac{1}{2})4^n}\le   2 \cdot 4^{-n}\le  c_24^{-n}$. Then, the estimate
 $g_n\le c_24^{-n}$ yields
\[
g_{n+1}(t)\le c_24^{-2n}\int_0^t c_24^{-n}(c_24^{-n}+e^{-{\frac{1}{2}}4^n})d\tau
\le c_24^{-2n}\int_0^t 2(c_2)^2 4^{-2n}d\tau\le 2c^3_2 t4^{-4n}.
\]
Thus,
\begin{equation}\label{sov5}
t_{n+1}\ge 4^{3n-1}/(2 c^2_2), \quad {\rm for\,} n\in 
\mathbb{N}.
\end{equation}
An upper bound on $t_n$ can be obtained as follows. Since $g_n\ge 0$ on $t\in [0,t_n]$ and $g_{n+1}=c_24^{-(n+1)}$ for $t\ge t_{n+1}$, 
we have   \begin{equation}t_{n+1}\le t_n+4^{3n-1}/(c^2_2), \, n\in \mathbb{N}\,.\label{sov12}\end{equation} Indeed,
this inequality holds trivially if $t_{n+1}\le t_n$. For the case  $t_{n+1}\ge t_n$, we have
\[
c_2 4^{-(n+1)}=g_{n+1}(t_{{n+1}})=M_n(g_n)(t_{n+1})\ge c_2 4^{-2n}\int_{0}^{t_{n+1}}g^2_n(\tau) dt
\ge c_2 4^{-2n}\int_{t_n}^{t_{n+1}} c_2^2 4^{-2n}dt.
\]
Summing up \eqref{sov12} in $n$, we get
\[
t_n\le t_1+\sum_{k=1}^n 4^{3k-4}/( c^2_2)\le t_1+4^{3n-3}/(c^2_2)\,.
\]
Since $t_1=(4c_1(c_1+e^{-\frac 12}))^{-1}$, the last estimate and \eqref{sov5} imply that there are positive constants $c_5$ and $c_6$ so that
\begin{equation}\label{sov6}
c_5 4^{3n}\le t_n\le  c_6 4^{3n}, \quad n\in \mathbb{N}\,.
\end{equation}




$\bullet$ \quad Let $n\ge 1$. Since  $g_k$ is non-decreasing in $t$, 
we can write the following bound for every $j\in \mathbb{Z}^+$:
\begin{equation*}\begin{split}\label{sov7}
g_{(n+j)+1}(t_n)&\le   c_2 4^{-2(n+j)} \int_0^{t_n}g_{(n+j)}(\tau)\cdot(g_{(n+j)}(\tau)+e^{-{\frac{1}{2}}4^{(n+j)}}) d\tau\\
&\le c_2 4^{-2(n+j)} g_{(n+j)}(t_n)\cdot(g_{(n+j)}(t_n)+e^{-{\frac{1}{2}}4^{(n+j)}})\cdot t_n\\
&\le c_2 4^{-2(n+j)} g_{(n+j)}(t_n)\cdot(g_{(n+j)}(t_n)+e^{-{\frac{1}{2}}4^{(n+j)}})\cdot c_6 4^{3n}\,,
\end{split}\end{equation*}
where we used \eqref{sov6} to bound $t_n$ in the last inequality. \bigskip

For shorthand,  let's denote
$a_{n,j}:=g_{n+j}(t_n)$ for $  j\ge 0, n\in \mathbb{N}$. Then, we can write
\[ a_{n,(j+1)}\le  \min \Bigl( c_2 4^{-2(n+j)} a_{n,j}\cdot(a_{n,j}+e^{-{\frac{1}{2}}4^{(n+j)}})\cdot c_6 4^{3n},   c_2 4^{-(n+j+1)} \Bigr)\,. \]
Notice also that $a_{n,0}=c_2 4^{-n}$.  The induction argument gives $a_{n,j}\le b_{n,j}$ where $\{b_{n,j}\}$ are introduced in \eqref{pop} a few lines below. 
Since $g_{n+j}(t)\le a_{n,j}$ for all $t\le t_n$ and $t_n\ge c_54^{3n}$, the estimate \eqref{main_lemma_technical} now follows from Proposition \ref{kur}. 
The proof is finished.

\end{proof}

Let $c_2,c_6$ be positive constants and $c_2\ge 2$. For each $n\in \mathbb{N}$, define $\{b_{n,j}\}_{j=0}^\infty$ recursively by
$b_{n,0}:=c_24^{-n}$ and
 \begin{equation}\label{pop}
  b_{n,(j+1)}= \min \Big(c_2 4^{-2(n+j)} b_{n,j}\cdot(b_{n,j}+e^{-{\frac{1}{2}}4^{(n+j)}})\cdot c_6 4^{3n},c_2 4^{-(n+j+1)}\Big)\quad \mbox{ for } j\ge 0. \end{equation}

\begin{proposition} \label{kur} There are  positive constants $c_3, c_4$ and $n_0\in \mathbb{N}$, such that
\[
 b_{n,j}\le c_3{4^{-n-c_42^j}}
  \quad\mbox{ for all }  n\ge n_0, j\in \mathbb{Z}^+\,.
\]
\end{proposition}
\begin{proof}
For a later use, take  $n_0\in \mathbb{N}$ so large that
\begin{equation}\label{n_0}
2c_2c_64^{3n_0}\ge 4.
\end{equation}
From now on, let $n\ge n_0$.
We first claim that
\begin{equation}\label{sov8}
b_{n,j}\ge   e^{-{\frac{1}{2}}4^{(n+j)}}
\end{equation}
for all $j\ge 0$ and for all $n\ge n_0$. Indeed, it can be shown by an induction in $j$: we know
$$b_{n,0}=c_24^{-n}\ge 2\cdot 4^{-n}\ge e^{-{\frac{1}{2}}4^{n}}  $$ because $
c_2\ge 2$.
Suppose
$b_{n,j}\ge  e^{-{\frac{1}{2}}4^{(n+j)}}$ for some $j\ge0$. We need to show $$b_{n,(j+1)}\ge  e^{-{\frac{1}{2}}4^{(n+j+1)}}
.$$
Recall that $b_{n,(j+1)}$ is either $\left(c_2 4^{-2(n+j)} b_{n,j}\cdot(b_{n,j}+e^{-{\frac{1}{2}}4^{(n+j)}})\cdot c_6 4^{3n}\right)$ or $\left( c_2 4^{-(n+j+1)}\right)$.
In the former case,
\begin{equation*}\begin{split}
b_{n,(j+1)}&\ge 2 c_2 c_6 4^{3n} (4^{-(n+j)})^2(e^{-{\frac{1}{2}}4^{(n+j)}})^2
\ge 2 c_2 c_6 4^{3n_0}      (4^{-(n+j)})^2(e^{-{\frac{1}{2}}4^{(n+j)}})^2\\
&\ge 4  (4^{-(n+j)})^2(e^{-{\frac{1}{2}}4^{(n+j)}})^2
\ge   (  e^{-{\frac{1}{2}}4^{(n+j)}})^2(e^{-{\frac{1}{2}}4^{(n+j)}})^2=  e^{-{\frac{1}{2}}4^{(n+j+1)}}\,,
\end{split}\end{equation*} where we use \eqref{n_0} for the third inequality.
If, however, $b_{n,(j+1)}=c_2 4^{-(n+j+1)}$, then
$b_{n,(j+1)}\ge 2\cdot 4^{-(n+j+1)} \ge e^{-{\frac{1}{2}}4^{(n+j+1)}}$. Thus, \eqref{sov8} is proved.\\ 

By the claim, for  all $n\ge n_0$ and $j\ge 0$,  we get \[ b_{n,(j+1)}\le   \min \Big(2c_2 4^{-2(n+j)} (b_{n,j})^2\cdot c_6 4^{3n},c_2 4^{-(n+j+1)}\Big). \]
To get the needed bound on $b_{n,j}$, we again argue by comparison to exact recursion.
Define $\{c_{n,j}\}_{j=0}^\infty$ by
$c_{n,0}:=b_{n,0}=c_24^{-n}$ and
by the following iteration: \[ c_{n,(j+1)}= \min \Big(2c_2 4^{-2(n+j)} (c_{n,j})^2 \cdot c_6 4^{3n},c_2 4^{-(n+j+1)}\Big) \quad \mbox{ for } j\ge 0. \]
Then, we clearly have $c_{n,j}\ge b_{n,j}$ for $j\ge 0$ and for   $n\ge n_0$.\\

To iterate the formula for $c_{n,j}$, it is convenient to rewrite it in the following form
 \[ c_{n,(j+1)}= \min \Big(  4^\beta 4^{n-2j} (c_{n,j})^2 , 4^\alpha  4^{-(n+j+1)}\Big) \quad \mbox{ for } j\ge 0\,, \]
where real $\alpha$ and $\beta$ are defined by $2 c_2c_6=4^\beta$ and $c_2=4^\alpha$. If we represent $c_{n,j}$ as  $c_{n,j}=4^{-p_{n,j}}$,
then $p_{n,0}=n-\alpha$ and
$$p_{n,(j+1)}= \max\Big( -\beta-n+2j+2p_{n,j},-\alpha+n+j+1\Big) \quad \mbox{ for } j\ge 0.$$
We further write $p_{n,j}=n+q_{n,j}$ and notice that
$$
q_{n,(j+1)}= \max\Big( -\beta+2(j+q_{n,j}) ,
-\alpha +j+1 \Big)  \quad \mbox{ for } j\ge 0.
$$
Take the smallest $j_0\in\mathbb{N}$ for which
$j_0\ge  \alpha+1$ and $2j_0\ge \beta$.
Then, we have $$q_{n,j_0}
\ge -\alpha + (j_0-1)+1
\ge 1
$$ and, for any $j\ge j_0$,
$$
q_{n,(j+1)}\ge  -\beta+2(j+q_{n,j})\ge  -\beta+2j_0+2q_{n,j} \ge 2q_{n,j}.
$$
It implies that, for any $j\ge j_0$, we get
$$ q_{n,j}\ge 2^{j-j_0}$$ and then $$  p_{n,j}\ge n+2^{j-j_0}.$$
In other words, for any $n\ge n_0$ and for any $j\ge j_0$,
we have
\begin{equation}c_{n,j}\le  4^{-(n+2^{j-j_0})}
=4^{-(n+(2^{-j_0})2^{j})}.\label{sov9}
\end{equation}
\\
We claim now that for any $n\ge n_0$ and for any $j\ge 0$,
one has
\begin{equation}\label{sov99}
c_{n,j}\le  4^{\alpha +(2^{-j_0}) } 4^{-(n+(2^{-j_0})2^{j})}.
\end{equation}
Indeed, for $j\ge j_0$, this follows from $\alpha\ge 0$ and \eqref{sov9}.
 From the definition of $c_{n,j}$, we get
 $ c_{n,j}\le   4^\alpha  4^{-(n+j)}$   for any $j\ge 0.$
 Thus, the case $j=0$ is trivial.
For $1\le  j\le  j_0$, we use an elementary bound $4^{-j}\le  4^{-(2^{-j_0})2^{j}}.$\\

Taking $c_3:=4^{\alpha +(2^{-j_0}) }$ and
$c_4:=2^{-j_0}$ in \eqref{sov99}, we finish the proof of the proposition. 
\end{proof}

{\Large \section*{acknowledgement}}
The work of KC was supported by NRF-2015R1D1A1A01058614, by the POSCO Science Fellowship of POSCO TJ Park Foundation, and by
 the Research Fund (1.170045.01) of UNIST(Ulsan National Institute of Science \& Technology). The work of SD on the first three sections
was supported by RSF-14-21-00025 and his research on the rest of the paper was supported by NSF Grant DMS-1464479.

\end{document}